\documentclass[11pt]{amsart}
\usepackage{amsmath}
\usepackage{amssymb, verbatim}
\usepackage{pstricks,pst-node,pst-plot}
\usepackage{comment}

\title[$C^{2,\alpha}$ estimates for twisted type equations]{$C^{2,\alpha}$ estimates for nonlinear elliptic equations of twisted type}
\author[T.C. Collins]{Tristan C. Collins}
\address{Department of Mathematics, Harvard University, 1 Oxford Street, Cambridge, MA 02138}
\email{tcollins@math.harvard.edu}

\theoremstyle{plain}
\newtheorem{thm}{Theorem}[section]
\newtheorem{prop}[thm]{Proposition}
\newtheorem{defn}[thm]{Definition}
\newtheorem{lem}[thm]{Lemma}

\theoremstyle{definition}
\newtheorem{ex}[thm]{Example}

\numberwithin{equation}{section}

\newcommand{\Tr}{\textrm{Tr}}

\newcommand{\del}{\partial}

\newcommand{\dbar}{\overline{\del}}
\newcommand{\ddb}{i\del\dbar}

\renewcommand{\leq}{\leqslant}
\renewcommand{\geq}{\geqslant}

\renewcommand{\epsilon}{\varepsilon}
\renewcommand{\phi}{\varphi}

\begin{document}
\begin{abstract}
 We prove a priori interior $C^{2,\alpha}$ estimates for solutions of fully nonlinear elliptic equations of twisted type.  For example, our estimates apply to equations of the type convex + concave.  These results are particularly well suited to equations arising from elliptic regularization.  As application, we obtain
 a new proof of an estimate of Streets-Warren \cite{SW} on the twisted real Monge-Amp\`ere equation.
 \end{abstract}

\maketitle

\section{introduction}
In this paper we will derive an {\em a priori} interior $C^{2,\alpha}$ estimate for solutions of the nonlinear, uniformly elliptic equation
\begin{equation}\label{eq: intro eq}
F(D^{2}u) =0
\end{equation}
under the assumption that $F$ is of {\em twisted type} (see Defintion~\ref{defn: twisted type}, below).  Higher order estimates for fully nonlinear elliptic equations are a subtle matter, and depend heavily on the structure of $F$.  The fundamental result in this vein is the famous estimate of Evans \cite{E} and Krylov \cite{Kr1, Kr2}, which says that if $F$ is convex, or concave, then the $C^{\alpha}$ norm of $D^2u$ on the interior is controlled by $L^{\infty}$ norm of $D^2u$.  Nadirashvili-Vl\u adu\c t \cite{NV} have produced counterexamples to Evans-Krylov type estimates for general fully nonlinear equations.  As a result, it is an interesting and important problem to understand when Evans-Krylov type estimates hold for nonlinear equations which are neither convex nor concave.  

Let us briefly recount the positive results currently available in the literature.  Caffarelli-Cabr\'e \cite{CC1} proved an Evans-Krylov type estimate for functions $F$ which are the minimum of convex and concave functions-- in particular, their result applies to equations of Issac's type.  Caffarelli-Yuan \cite{CY} proved $C^{2,\alpha}$ a priori estimates under the assumption that level set  $\{F=0\}$ has uniformly convex intersection with a family of planes.  In essence, this allows one of the principle curvatures of $\{F=0\}$ to be negative.  Yuan \cite{Y} proved higher regularity for the 3 dimensional special Lagrangian equation.  More recently, motivated by the pluriclosed flow introduced by Streets-Tian \cite{St, ST1, ST2, ST3}, Streets-Warren \cite{SW} exploited the partial Legendre transform to prove an Evans-Krylov type theorem for the {\em real twisted Monge-Amp\`ere equation} 
\begin{equation}\label{eq: sw}
\log\det u_{x_{i}x_{j}} - \log\det(-u_{y_{k}y_{l}}) =0,
\end{equation}
where $u(x,y)$ is assumed to be uniformly convex in the $x$ variables, and uniformly concave in the $y$ variables.  The authors also consider parabolic and complex analogs of this equation.  

 The concern of this note is to extend the Evans-Krylov estimate to a general class of nonconvex equations of twisted type.  Let us give the precise setting for our result.  Let $S^{n\times n} \subset M^{n\times n}$ be the set of symmetric $n\times n$ matrices with real entries.  Suppose that $u$ is a smooth solution of equation~\eqref{eq: intro eq}, and let $\Sigma = \{ F=0 \} \subset S^{n\times n}$.
\begin{defn}\label{defn: twisted type}
We say that the uniformly elliptic equation~\eqref{eq: intro eq} is of {\em twisted type}, if it can be written as
\begin{equation}\label{eq: main eq}
F_{\cup}(D^{2}u) + F_{\cap}(D^2 u) = 0.
\end{equation}
where $F_{\cup}, F_{\cap}$ satisfy the following structural conditions;
\begin{enumerate} 
\item[(S1)] $F_{\cup}, F_{\cap}$ are continuous (possibly degenerate)  elliptic operators on $D^{2}u(B_{1})$.
\item[(S2)] There exists an open set $\mathcal{O}$ containing $\overline{D^{2}u(B_{1})}$ such that $F_{\cup}, F_{\cap}$ are $C^{2}$ on $\mathcal{O}$.
\item[(S3)] $F_{\cup}$ is convex and uniformly elliptic.
\item[(S4)] $F_{\cap}$ is weakly concave, in the sense of Definition~\ref{defn: wk concave} below.
\end{enumerate}
\end{defn}
\begin{defn}\label{defn: wk concave}
We say that $F_{\cap}$ is weakly concave if there exists a subset $U \subset \mathbb{R}$ and a function $G : U \rightarrow \mathbb{R}$ which is continuous, and smooth on the interior of $U$, such that 
\begin{enumerate}
\item[(i)] $\overline{F_{\cap}(\Sigma)} \subset U$,
\item[(ii)] $G' >0, G'' \leq 0$ and $G(F_{\cap}(\cdot))$ is concave,
\item[(iii)] There exists a continuous function $Q(x) : \mathbb{R} \rightarrow \mathbb{R}_{>0}$ such that $G'(x) \geq Q(x)$ for $x \in{\rm int} U$.
\end{enumerate}
\end{defn}

We have endeavored to make Definition~\ref{defn: twisted type} as broad as possible, which unfortunately necessitates its somewhat technical appearance.  In many cases, the conditions of Definition~\ref{defn: twisted type} are easily seen to be satisfied, as we will discuss in the examples below. Our main theorem is that equations of this form admit interior $C^{2,\alpha}$ a priori estimates.

\begin{thm}\label{thm: main}
Suppose $u$ is a smooth solution of equation~\eqref{eq: intro eq} on $B_{1} \subset \mathbb{R}^{n}$.  Assume that $F$ is of twisted type.  Assume moreover that $F_{\cap}(B_{1}) \subset {\rm int} U$.  Let  $0<\lambda < \Lambda < \infty$ be ellipticity constants for both $F, F_{\cup}$.  For every $0 < \alpha <1$ we have the estimate
\begin{equation*}
\|D^{2}u\|_{C^{\alpha}(B_{1/2})} \leq C\left(n,\lambda, \Lambda, \alpha, \gamma, \Gamma, F_{\cup}(0), \|D^{2}u\|_{L^{\infty}(B_{1})}\right),
\end{equation*}
where 
\begin{equation*}
\begin{aligned}
0<\gamma &:= \inf_{x\in F_{\cup}(D^{2}u)(B_{1})} Q'(-x) \\
\Gamma &:= {\rm Osc}_{B_1} G(-F_{\cup}(D^2u))
\end{aligned}
\end{equation*}
depend only on $ \Lambda, \|D^2u\|_{L^{\infty}(B_{1})}$ and the functions $G, Q$ of Definition~\ref{defn: wk concave}.
\end{thm}

Twisted type equations arise naturally in problems in differential geometry; let us describe some examples.

\begin{ex}
Suppose $F_{\cup}$ is a smooth convex, uniformly elliptic operator, and $F_{\cap}$ is a smooth concave degenerate elliptic operator.  Then the operator $F=  F_{\cup} + F_{\cap}$ is of twisted type, with $G(x) = x, Q(x) =1$.  In particular, $U= \mathbb{R}$ in Definition~\ref{defn: wk concave}, and hence the condition $F_{\cap}(B_{1}) \subset {\rm int} U$ is vacuous. 
\end{ex}

Let us consider a slightly less trivial, and more explicit family of examples.

\begin{ex}\label{ex: twisted sigma}
For a symmetric matrix $M$, we let $\sigma_{k}(M)$ denote the $k$-th symmetric polynomial of the eigenvalues of $M$.  It is well-known that $\sigma_{k}(D^{2}u)$ is an elliptic operator on the cone of $k$-convex functions.  Furthermore, if we take $G(x) = x^{1/k}$ for $x\geq 0$, then $\sigma_{k}(D^{2}u)^{1/k}$ is concave, and $G$ satisfies properties $(i)$ and $(ii)$ of Definition~\ref{defn: wk concave}.  For item $(iii)$ we can take
\begin{equation*}
Q(x) = \min\{1, k|x|^{\frac{1}{k}-1}\}
\end{equation*}
In particular, if we assume $\|D^{2}u\|_{L^{\infty}}(B_{1}) \leq \Lambda$, then the equations
\begin{equation}\label{eq: twist k}
\Delta u + \sigma_{k}(D^{2}u) = 1
\end{equation}
are uniformly elliptic for $u\in \Gamma_{k}(B_{1})$, and are all of twisted type.  The estimate in Theorem~\ref{thm: main} applies provided $\sigma_{k}(x) >0$ for all $x \in B_{1}$.  These operators do not have convex sublevel sets, and hence the Evans-Krylov theorem does not apply.  Moreover, the result of Caffarelli-Yuan \cite{CY} does not apply as the convexity of the level sets $\{\sigma_{k} =t \}$ degenerate as $t \rightarrow 0$.   Finally, let us remark that if $k=n$, then the assumption that $\sigma_n >0$ is superfluous thanks to the constant rank theorem of Bian-Guan \cite{BG}.
\end{ex}

Equations of the form considered in Example~\ref{ex: twisted sigma} arise naturally, both in geometry and analysis.  For example, the $k=n$ case of~\eqref{eq: twist k}, arose in the work of the author and G. Sz\'ekelyhidi on the $J$-flow on toric K\"ahler manifolds \cite{CS}.  Previously, Fu-Yau \cite{FY} solved a twisted Monge-Amp\`ere type equation similar to (but more complicated than)~\eqref{eq: twist k} in their study of the Strominger system.  Fu-Yau succeeded in bypassing $C^{2,\alpha}$ estimates by estimating directly the $C^{3}$ norm by a difficult maximum principle argument.
 
 One might argue that instead of equation~\eqref{eq: twist k}, one should instead study the concave equation
 \begin{equation}\label{eq: alt eq}
 \Delta u + \sigma(D^{2}u)^{1/k} =1.
 \end{equation}
 Let us provide a brief rebuttal to this point.  First,  note that~\eqref{eq: alt eq} is not uniformly elliptic without a strict lower bound $\sigma_{k}>\lambda >0$, while equation~\eqref{eq: twist k} is uniformly elliptic without any further assumptions.  For this reason the family of equations~\eqref{eq: twist k} arises as a natural ``elliptic regularization" of the $k$-Hessian equation.  Furthermore, on a K\"ahler manifold $X$ of dimension $n$ it is natural to consider equations of the form 
 \begin{equation*} 
 \omega^{n-1} \wedge (\alpha + \ddb\phi )+ \omega^{n-k}\wedge (\alpha + \ddb \phi)^{k} = c \omega^{n}
 \end{equation*}
where $\omega$ is a fixed K\"ahler form, $\alpha$ is a real, smooth, closed $(1,1)$-form, and $\phi$ is unknown.  In this case, the constant $c$ is determined cohomologically by the classes $[\omega], [\alpha] \in H^{1,1}(X,\mathbb{R})$.  In the toric case, this equation reduces to equation~\eqref{eq: twist k}.  On the other hand, for the concave equation whose local expression is~\eqref{eq: alt eq},  the right hand side is not determined cohomologically and one needs additional information, often unavailable in geometric problems, to specify the equation completely.

Our last example illustrates the necessity of the technicalities in Definition~\ref{defn: twisted type}.
 \begin{ex}\label{ex: sw}
Let $\{x_{1}, \dots, x_{k}\}$ be standard coordinates on $\mathbb{R}^{k}$, and let $\{y_{1}, \dots, y_{\ell}\}$ be standard coordinates on $\mathbb{R}^{\ell}$, so that we may identify $\mathbb{R}^{k+\ell} = \mathbb{R}^{k} \times \mathbb{R}^{\ell}$.  Set $n= k+\ell$.  Our identification induces an decomposition of any matrix $M \in S^{n\times n}$ as
\[
M = M_{kk} \oplus M_{\ell k} \oplus M_{\ell k}^{T} \oplus M_{\ell\ell}
\]
where $M_{kk}$ is the symmetric $k\times k$ matrix corresponding to the $\mathbb{R}^{k}$ factor in the obvious way, and similarly for $M_{\ell \ell}$.  Define maps by 
\[
\pi_{k}(M) = M_{kk}, \qquad \pi_{\ell}(M) = M_{\ell}.
\] For $\kappa, \lambda >0$, we set
\begin{equation}
\mathcal{E}_{\lambda, \kappa} := \left\{ M \in S^{n\times n} :
\begin{aligned} 
&(i)\, \lambda I_{k} \leq \pi_{k}(M) \leq \lambda^{-1} I_{k}, \quad \\
&(ii)\,  \lambda I_{\ell} \leq -\pi_{\ell}(M) \leq \lambda^{-1}I_{\ell},\\
&(iii)\, \|M\| \leq \kappa
\end{aligned} \right\}
\end{equation}
where $I_{j}$ denotes the $j\times j$ identity matrix for $j=k,\ell$.  Note the $\mathcal{E}_{\lambda, \kappa}$ is a compact, convex subset of $S^{n\times n}$.

Consider the real twisted Monge-Amp\`ere equation of Streets-Warren \cite{SW}, given in equation~\eqref{eq: sw}.  Let $u$ be a smooth solution of~\eqref{eq: sw} on $B_{1}$ and suppose that $\overline{D^{2}u(B_{1})} \subset \mathcal{E}_{\lambda, \kappa}$.  Write equation~\eqref{eq: sw} in the form
 \begin{equation*}
  -\left(\det (-D^{2}_{y}u)\right)^{1/n} + \epsilon \Tr \left(D^{2}_{x}u\right) + \left( \det (D^{2}_{x}u)\right)^{1/n} -  \epsilon \Tr \left(D^{2}_{x}u\right) =0.
 \end{equation*}
 We note that under the above assumptions
 \begin{equation*}
 F_{\cup}(M) :=    -\left(\det (-\pi_{\ell}M)\right)^{1/n} + \epsilon \Tr \left(\pi_{k}M\right)
 \end{equation*}
 is smooth and uniformly elliptic on $\mathcal{E}_{\lambda, \kappa}$ for $\epsilon >0$ and convex.  Furthermore 
 \begin{equation*}
 F_{\cap} (D^{2}u) := \left( \det (D^{2}_{x}u)\right)^{1/n} -  \epsilon \Tr \left(D^{2}_{x}u\right)
 \end{equation*}
 is degenerate elliptic for $\epsilon>0$ depending only on $k,\ell, \lambda$, and is concave.  However, we note that $F_{\cup}$ is {\em not} uniformly elliptic outside of $\mathcal{E}_{\lambda, \kappa}$.  To remedy this we use an envelope trick exploited by Wang \cite{W}, and Tosatti-Wang-Weinkove-Yang \cite{TWWY} to extend $F_{\cup}$ to a uniformly elliptic convex operator outside of $\mathcal{E}_{\lambda, \kappa}$.  To be precise, we set
 \begin{equation}
 \widetilde{F_{\cup}}(M) := \sup_{N, c}\left \{{\rm Tr} \left(-\pi_{\ell}(N) \pi_{\ell}(M) \right) + {\rm Tr}(\pi_{k}(M)) +c \right\}
 \end{equation}
 where the supremum is take over all $N\in \mathcal{E}_{\lambda, \kappa}, c \in \mathbb{R}$ such that
 \begin{equation}
 {\rm Tr} \left(-\pi_{\ell}(N) \pi_{\ell}(M) \right) +c + {\rm Tr}(\pi_{k}(M)) \leq F_{\cup}(X)\qquad \forall X \in \mathcal{E}_{\lambda, \kappa}
\end{equation}

The operator $ \widetilde{F_{\cup}}(\cdot)$ is precisely the convex envelope of the graph of $F_{\cup}(\cdot)$ over $\mathcal{E}_{\lambda, \kappa}$.  In particular, we immediately obtain that $\widetilde{F_{\cup}}(\cdot)$ is continuous, convex and agrees with $F_{\cup}(\cdot)$ on the convex set $\mathcal{E}_{\lambda, \kappa}$.  It suffices to check that $\widetilde{F_{\cup}}(\cdot)$ is uniformly elliptic.  We include the short argument for the readers' convenience (see also \cite{W}).  Fix $P \in S^{n\times n}$, with $P \geq 0$.   By compactness there exist $N_{1}, N_{2} \in \mathcal{E}_{\lambda, \kappa}$ and $c_{1}, c_{2} \in \mathbb{R}$ such that
\begin{equation}\label{eq: tilF eq}
\begin{aligned}
\widetilde{F_{\cup}}(M+P) &= {\rm Tr} \left(-\pi_{\ell}(N_{1}) \pi_{\ell}(M+P) \right) + {\rm Tr}(\pi_{k}(M+P)) +c_{1}\\
\widetilde{F_{\cup}}(M) &= {\rm Tr} \left(-\pi_{\ell}(N_{2}) \pi_{\ell}(M) \right) + {\rm Tr}(\pi_{k}(M)) +c_{2}
\end{aligned}
\end{equation}
By maximality we must have
\begin{equation}\label{eq: tilF bd}
\begin{aligned}
\widetilde{F_{\cup}}(M+P) &\geq  {\rm Tr} \left(-\pi_{\ell}(N_{2}) \pi_{\ell}(M+P) \right) + {\rm Tr}(\pi_{k}(M+P)) +c_{2}\\
\widetilde{F_{\cup}}(M) &\geq  {\rm Tr} \left(-\pi_{\ell}(N_{1}) \pi_{\ell}(M) \right) + {\rm Tr}(\pi_{k}(M)) +c_{1}
\end{aligned}
\end{equation}
Combining the first line of~\eqref{eq: tilF bd} with the second line of~\eqref{eq: tilF eq} we get
\begin{equation*}
\widetilde{F_{\cup}}(M+P)-\widetilde{F_{\cup}}(M) \geq  {\rm Tr}\left(-\pi_{\ell}(N_{2}) \pi_{\ell}(P) \right) + {\rm Tr}(\pi_{k}(P)) \geq \lambda \|P\|.
\end{equation*}
Similarly, combining the second line of ~\eqref{eq: tilF bd} with the first line of~\eqref{eq: tilF eq} we get
\begin{equation*}
\widetilde{F_{\cup}}(M+P)-\widetilde{F_{\cup}}(M) \leq {\rm Tr}\left(-\pi_{\ell}(N_{1}) \pi_{\ell}(P) \right) + {\rm Tr}(\pi_{k}(P)) \leq \lambda^{-1} \|P\|.
\end{equation*}
 In particular, the real twisted Monge-Amp\`ere equation~\eqref{eq: sw} is of twisted type in the sense of Definition~\ref{defn: twisted type}, and so Theorem~\ref{thm: main} gives a new proof of an estimate of Streets-Warren \cite[Theorem 1.1]{SW}. 
  \end{ex}

 The proof of Theorem~\ref{thm: main} hinges on the observation that $G(F_{\cap}(D^{2}u))$ is a supersolution of the linear equation $L\phi=0$, where $L$ denotes the linearization of~\eqref{eq: intro eq}; see Lemma~\ref{lem: key lem} below.  With this observation we can employ the machinery used in the proof of the Evans-Krylov theorem with some modifications to account for the fact that we must work with general elliptic operators rather than the Laplacian.  We take this up in detail in section 2.  In section 3 we provide a brief description of some applications and extensions.

\section{Proof of Theorem~\ref{thm: main}}
For ease of notation, whenever $F_{\cup}, F_{\cap}$ are $C^{2}$, we will set
\begin{equation}
F_{\alpha}^{ij} := \frac{\del F_{\alpha}}{\del a_{ij}}, \quad F_{\alpha}^{ij,rs} := \frac{\del^{2} F_{\alpha}}{\del a_{ij} \del a_{rs}}
\end{equation}
for $\alpha = \cup, \cap$ and $1 \leq i,j,r,s \leq n$.
We begin with some simple calculations.  First, if $u$ solves equation~\eqref{eq: main eq}, then by structural condition (S2) we have
\begin{eqnarray}
Lu_{a} := \left(F_{\cup}^{ij} + F_{\cap}^{ij}\right)u_{aij} &=& 0\\
\left(F_{\cup}^{ij}+ F_{\cap}^{ij}\right)u_{abij} + \left(F_{\cup}^{ij,rs} + F_{\cap}^{ij,rs}\right)u_{aij}u_{brs}&=&0
\end{eqnarray}
The main observation is contained in the following lemma.

\begin{lem}\label{lem: key lem}
Under the assumptions of Theorem~\ref{thm: main}, $G(F_{\cap}(D^2u))$ is a supersolution of the linearized equation on $B_{1}$.
\begin{equation*} 
L[G(F_{\cap}(D^2u))] \leq 0.
\end{equation*}
\end{lem}
\begin{proof}
By structural condition (S2), $F_{\cup}, F_{\cap}$ are $C^{2}$ in an open neighborhood of $D^{2}u(B_{1})$.  Since $F_{\cap}(D^{2}u(B_{1})) \subset {\rm int}U$, the function $G(F_{\cap}(D^2u))$ is $C^2$ in $B_{1}$.  Hence, we can compute
\begin{equation*}
\begin{aligned}
\del_{a} G(F_{\cap}(D^2u)) &= G' F_{\cap}^{ij} u_{aij}\\
\del_{b}\del_{a} G(F^{\cap}(D^2u)) &= G' F_{\cap}^{ij} u_{abij} + \left(G' F_{\cap}^{ij,rs}  + G'' F_{\cap}^{ij}F_{\cap}^{rs}\right) u_{aij} u_{brs}.
\end{aligned}
\end{equation*}
As a result, we obtain
\begin{equation*}
\begin{aligned}
L G(F_{\cap}(D^2u)) =& G' F_{\cap}^{ab} L(u_{ab}) +  \left(G' F_{\cap}^{ij,rs}  + G'' F_{\cap}^{ij}F_{\cap}^{rs}\right)\left(F_{\cup}^{ab} + F_{\cap}^{ab}\right) u_{aij} u_{brs}\\
=& \left[-G'F_{\cap}^{ab}\left(F_{\cup}^{ij,rs} + F_{\cap}^{ij,rs}\right)\right]u_{aij} u_{brs}\\
& + \left[\left(G' F_{\cap}^{ij,rs}  + G'' F_{\cap}^{ij}F_{\cap}^{rs}\right)\left(F_{\cup}^{ab} + F_{\cap}^{ab}\right)\right]  u_{aij} u_{brs} \\
=&F_{\cup}^{ab}\left(G' F_{\cap}^{ij,rs}  + G'' F_{\cap}^{ij}F_{\cap}^{rs}\right)u_{aij} u_{brs} + G''F_{\cap}^{ab}F_{\cap}^{ij}F_{\cap}^{rs} u_{aij}u_{brs}\\
& -G'F_{\cup}^{ij,rs}u_{aij} u_{brs}
\end{aligned}
\end{equation*}
Now, since $G(F_{\cap})$ is concave, and $F_{\cup}$ is elliptic, the first term is non-positive.  Moreover, since $G'' \leq 0$, and $F_{\cap}$ is (degenerate) elliptic, the second term is also non-positive.  Finally, since $G'>0$, and $F_{\cup}$ is convex, the third term is non-positive as well.  As a result, we have
\begin{equation*}
LG(F_{\cap}(D^2u)) \leq 0
\end{equation*}
as was to be proved.
\end{proof}

The following proposition is essential, and is based on the proof of the Evans-Krylov theorem (cf. \cite[Proposition 1]{CY}, \cite{CC}).

\begin{prop}\label{prop: main EK}
Let the assumptions of Theorem~\ref{thm: main} be in force.  Then, for any $\epsilon >0$, there exists a positive constant $\eta = \eta(n,\lambda, \Lambda, \epsilon, \gamma, \Gamma, \|D^{2}u\|_{L^{\infty}(B_{1})})$ and a quadratic polynomial $P$ so that, for all $x \in B_{1}$, we have 
\begin{equation*}
\begin{aligned}
\left| \frac{1}{\eta^{2}} u(\eta x) - P(x)\right| &\leq \epsilon \\
F(D^{2}P) &=0.
\end{aligned}
\end{equation*}
\end{prop}
\begin{proof}
Fix constants $\rho, \xi, \delta, k_{0} >0$ to be determined.  We let $C$ denote a constant depending on the stated data which may change from line to line.  Define 
\begin{equation*}
\begin{aligned}
t_{k} &:= \sup_{B_{1/2^{k}}} F_{\cup}(D^2u), \qquad 1 \leq k \leq k_{0}\\
s_{k} &:= \inf_{B_{1/2^{k}}} G(F_{\cap}(D^2u)), \qquad 1 \leq k \leq k_{0}.
\end{aligned}
\end{equation*}
Note that $s_{k} = G(-t_{k})$ since $G$ is increasing.   

Consider the set
\begin{equation*}
E_{k} := \{ x\in B_{1/2^k} | F_{\cup}(D^2u) \leq t_{k}-\xi \}.
\end{equation*}
Suppose that
\begin{equation}\label{eq: easy case}
 \exists \, 1\leq \ell \leq k_{0} \text{ so that } \quad |E_{k}| \leq \delta |B_{1/2^{\ell}}|.
\end{equation}
Let $w_{\ell}(x) = 2^{2\ell}u\left(\frac{x}{2^{\ell}}\right)$, and let $\upsilon$ solve the equation
\begin{equation*}
\left\{
\begin{array}{ll}
F_{\cup}(D^2 \upsilon(x)) = t_{\ell} &x \in B_{1}\\
\upsilon(x) = w_{\ell}(x) &x \in \del B_{1}.
\end{array}
\right.
\end{equation*}
Note that $\upsilon$ exists by standard elliptic theory (see, e.g. \cite[Chapter 9]{CC}).  Since $F_{\cup}$ is uniformly elliptic, the Alexandroff-Bakelman-Pucci (see, for instance, \cite[Theorem 2.21]{FH}, \cite[Theorem 3.2]{CC}) estimate implies
\begin{equation}\label{eq: ABP sup est}
\begin{aligned}
\|\upsilon - w_{\ell}\|_{L^{\infty}(B_{1})} &\leq C\|F_{\cup}(D^2\upsilon) - F_{\cup}(D^{2}w_{\ell})\|_{L^{n}(B_{1})} \\
&\leq C(n) (\xi^{n} + \delta)^{1/n}.
\end{aligned}
\end{equation}
Set $\hat{\upsilon} = \upsilon - w_{\ell}(0) - \langle \nabla w_{\ell}(0), x \rangle$, and observe that $F_{\cup}(D^{2} \hat{\upsilon}) = t_{\ell}$.  Since $F_{\cup}$ is uniformly elliptic and convex we can apply the usual Evans-Krylov Theorem \cite{E, Kr1, Kr2}, \cite[Theorem 6.6]{CC}, \cite{GT}, to find a constant $\beta = \beta(\lambda, \Lambda, n) \in (0,1)$ depending only on universal data such that
\begin{equation}\label{eq: EK est}
|\hat{\upsilon}|_{C^{2,\beta}(B_{1/2})} \leq C \left(\|\hat{\upsilon}\|_{L^{\infty}(B_{1})} + |F_{\cup}(0)-t_{\ell}| \right).
\end{equation} 
By the uniform ellipticity of $F_{\cup}$, we have
\begin{equation}\label{eq: sum est}
|t_{\ell} - F_{\cup}(0)| \leq \Lambda \|D^{2}u\|_{L^{\infty}(B_{1})}.
\end{equation}
Furthermore, applying the Alexandroff-Bakelman-Pucci estimate \cite[Theorem 3.2]{CC} to the uniformly elliptic equation $F_{\cup}(D^{2}\hat{\upsilon})= t_{\ell}$ we obtain
\begin{equation}\label{eq: ups ABP}
\begin{aligned}
\|\hat{\upsilon}\|_{L^{\infty}(B_{1})} &\leq C \left[\sup_{x\in \del B_{1}} |\hat{\upsilon}(x)| + |t_{\ell}|\right]\\
&= C \left[\sup_{x\in \del B_{1}} |w_{\ell}(x) -w_{\ell}(0) - \langle \nabla w_{\ell}(0), x \rangle| + |t_{\ell}|\right]\\
&\leq C \left(\|D^{2}w_{\ell}\|_{L^{\infty}(B_{1})} + |t_{\ell}| \right)\\
&= C \left(\|D^{2}u\|_{L^{\infty}(B_{1})} + |t_{\ell}|\right).
\end{aligned}
\end{equation}
Combining estimates ~\eqref{eq: EK est},~\eqref{eq: sum est} and~\eqref{eq: ups ABP} 
\begin{equation*}
\|D^{2}\upsilon\|_{C^{\beta}(B_{1/2})} = \|D^{2}\hat{\upsilon}\|_{C^{\beta}(B_{1/2})} \leq C\left(\|D^2u\|_{L^{\infty}(B_{1})} + |F_{\cup}(0)| \right).
\end{equation*}
Let $P$ denote the quadratic part of $\upsilon$ at the origin.  Then, for $x \in B_{1/2}$ we have
\begin{equation}\label{eq: poly approx 1}
\begin{aligned}
|w_{\ell}-P|(x) &\leq |w_{\ell} - \upsilon|(x) + |\upsilon - P|(x) \\
&\leq C(\xi^{n} + \delta)^{1/n} + \|D^{2}\upsilon\|_{C^{\beta}(B_{1/2})}|x|^{2+\beta} \\
&\leq C(\xi^{n} + \delta)^{1/n} + C |x|^{2+\beta}.
\end{aligned}
\end{equation}

Before proceeding, let us address the case when the assumption in~\eqref{eq: easy case} does not hold.  Namely, assume that
\begin{equation}\label{eq: hard case}
\forall 1\leq \ell \leq k_{0}, \quad |E_{\ell}| \geq \delta |B_{1/2^{\ell}}|.
\end{equation}
Define $w_{k}(x) = 2^{2k}u\left(\frac{x}{2^{k}}\right)$.  By Lemma~\ref{lem: key lem}, we have
\begin{equation*}
L[G(F_{\cap}(D^{2}w_{k}))- s_{k}] \leq 0.
\end{equation*}
We now apply the weak Harnack inequality repeatedly to show that $G(F_{\cap}(D^{2}w_{k}))$ must concentrate in measure near a level set, for $k_{0}$ sufficiently large, depending only on $\xi, \delta, \Gamma$.  This will imply that assumption~\eqref{eq: easy case} always holds.  

The weak Harnack inequality \cite[Theorem 4.8]{CC} implies that, for $x\in B_{1/2}$
\begin{equation*}
G(F_{\cap}(D^{2}w_{k}))(x)-s_{k} \geq c(n,\lambda) \| G(F_{\cap}(D^{2}w_{k})) -s_{k} \|_{L^{p_{0}}(B_{1})},
\end{equation*}
where $p_{0} = p_{0}(n,\lambda, \Lambda)$.  Since $G$ is concave and increasing, the definition of weak convexity implies that
\begin{equation*}
G(F_{\cap}(D^{2}w_{k})) \geq s_{k} +\gamma \xi \quad \forall x\in E_{k},
\end{equation*}
In particular, thanks to~\eqref{eq: hard case} we obtain
\begin{equation*}
s_{k+1} \geq s_{k} + c(n,\lambda) \gamma \xi \delta^{1/p_{0}} =: s_{k}+\theta.
\end{equation*}
It follows immediately that after 
\[
k_{0} = \frac{{\rm Osc}_{B_{1}} G(F_{\cap}(D^2u))}{\theta} =: \frac{\Gamma}{\theta}
\]
iterations it must hold that $|E_{k_{0}}| \leq \delta |B_{1/2^{k_{0}}}|$, for otherwise 
\[
s_{k_{0}+1} \geq \sup_{B_{1}}G(F_{\cap}(D^2u))
\]
which is absurd.  In particular, assumption~\eqref{eq: easy case} always holds provided we take $k_{0}$ large enough, depending on $\xi, \delta, \gamma, \Gamma$. 

In conclusion, by~\eqref{eq: poly approx 1}, there is an $1 \leq \ell \leq k_{0}$, and a polynomial $P$ such that
\begin{equation*}
|w_{\ell}-P|(x) \leq C(\xi^{n} + \delta)^{1/n} + C |x|^{2+\beta}.
\end{equation*}
Set $x = \rho y$, and $\hat{P}(y) = \rho^{-2}P(\rho y)$, then we have
\begin{equation*}
\left| \frac{1}{\rho^{2}} w_{\ell}(\rho y) - \hat{P}(y)\right| \leq C \frac{(\xi^{n} + \delta)^{1/n}}{\rho^{2}} + C \rho^{\beta}
\end{equation*}
for $|y| \leq 1$.  We may now apply \cite[Lemma 2]{CY} to obtain a polynomial $\tilde{P}$ such that $F(D^{2}\tilde{P}) =0 $, and 
\begin{equation*}
\left| \frac{1}{\rho^{2}} w_{\ell}(\rho y) - \tilde{P}(y)\right| \leq C \frac{(\xi^{n} + \delta)^{1/n}}{\rho^{2}} + C \rho^{\beta}.
\end{equation*}
We now choose $\rho$, then $\xi, \delta$ (which determine $k_{0}$), depending on $n, \lambda, \Lambda, \gamma, \epsilon$ such that
\begin{equation*}
\left| \frac{1}{\eta^{2}} u(\eta y) - \tilde{P}(y)\right| \leq \epsilon
\end{equation*}
where $\eta = \eta(n, \lambda, \Lambda, \gamma, \epsilon) = \rho/2^{\ell}$, and $1\leq \ell \leq k_{0}$.
\end{proof}

The conclusion of Theorem~\ref{thm: main} follows immediately from Proposition~\ref{prop: main EK}, either by appealing directly to Savin's small perturbations theorem~\cite[Theorem 1.3]{S}, or by following the argument in Caffarelli-Yuan \cite{CY}.

\section{Applications and Extensions}

In this section we indicate a few applications of Theorem~\ref{thm: main}.  The first application is a rigidity result (Lemma~\ref{lem: Lois}) for solutions of twisted type equations on $\mathbb{R}^{n}$.  As an application of this rigidity, we extend Theorem~\ref{thm: main} to a more general class of equations in  Theorem~\ref{thm: appl}, which is the main result of this section.  Theorem~\ref{thm: appl} may be of interest to geometers. 

\begin{lem}\label{lem: Lois}
Suppose $u: \mathbb{R}^{n} \rightarrow \mathbb{R}$ is a smooth function with $|D^{2}u|_{L^{\infty}(\mathbb{R}^{n})} \leq K < \infty$.  Suppose that $u$ solves the twisted type equation $F(D^{2}u) =0$ on $\mathbb{R}^{n}$.  Then $u$ is a quadratic polynomial.
\end{lem}
\begin{proof}
Without loss of generality we may assume that $u(0)=0$, $\nabla u(0) =0$.  Since $|D^{2}u| \leq K$, we have the estimate $|\nabla u(x)| \leq K|x|$. For $R >0$ we define
\[
v_{R}(x) := R^{-2}u(Rx).
\]
Then $v_{R}(x)$ satisfies $F(D^{2}v_{R}) =0$ on $B_{1}$, together with the bounds
\[
|\nabla v_{R}(x)|_{L^{\infty}(B_{1})} \leq \Lambda, \qquad |D^{2}v_{R}|_{L^{\infty}(B_{1})} \leq K.
\]
By Theorem~\ref{thm: main} we have an estimate $|D^{2}v_{R}|_{C^{\alpha}(B_{1/2})} \leq C$ for a constant $C$ independent of $R$. Writing this in terms of $u$ gives
\[
R^{\alpha} |D^{2}u|_{C^{\alpha}(B_{R/2})} \leq C.
\]
It follows that $D^{2}u$ is a constant, and hence $u$ is a quadratic polynomial.
\end{proof}

We can now prove an extension of Theorem~\ref{thm: main}.

\begin{thm}\label{thm: appl}
Suppose $u$ is a smooth solution of the uniformly elliptic equation
\begin{equation}\label{eq: Fx eq}
F(D^{2}u, x) =0, \qquad x \in B_{2}.
\end{equation}
Assume that $F(\cdot, \cdot)$ is smooth in both slots, and that $F(\cdot, x)$ is of twisted type for each $x \in B_{2}$.  Then for every $0<\alpha<1$ we have the estimate
\begin{equation*}
\|D^{2}u\|_{C^{\alpha}(B_{1/2})} \leq C
\end{equation*}
where $C$ depends on $F, \alpha, \|D^{2}u\|_{L^{\infty}(B_{2})}$.
\end{thm}

\begin{proof}
Set
\[
N_{u} = \sup_{B_{1}} d_{x} |D^{3}u(x)|, \qquad d_{x} = d(x, \del B_{1}),
\]
and assume that the supremum is achieved at some point $x_0 \in B_{1}$.  Define
\begin{equation}
\tilde{u}(z) = d_{x_{0}}^{-2} N_{u}^{2}u\left(x_{0} + d_{x_0}N_{u}^{-1} z \right) -A - A_{i}z_{i}.
\end{equation}
where $A, A_{i}$ are chose so that $\tilde{u}(0) =0$ and $ \nabla \tilde{u}(0)= 0$.  The function $\tilde{u}(z)$ is defined on $B_{N_{u}}(0)$, and satisfies
\[
D^{2}\tilde{u} = D^{2}u, \qquad \|D^{3}\tilde{u}\|_{L^{\infty}(B_{N_{u}(0)})} = |D^{3}\tilde{u}(0)| = 1.
\]
Furthermore, $\tilde{u}$ solves the equation
\begin{equation}
F(D^{2}\tilde{u}, x_{0} + d_{x_0}N_{u}^{-1} z ) =0, 
\end{equation}
on $B_{N_{u}}(0)$.

In order to prove the theorem, we use a contradiction argument, together with this rescaling.  Suppose there exists a sequence $u_{n}$ of functions on $B_{2}$ solving \eqref{eq: Fx eq}, and having a uniform bound with $\|D^{2}u_{n}\|_{L^{\infty}(B_{2})} \leq K$.  For the sake of obtaining a contradiction, we assume that the sequence $\{N_{n}\}$ increases to $+\infty$ .  Let $x_{n} \in B_{1}$ be the point where $N_{n}$ is achieved.  By passing to a subsequence (not relabeled) we may assume that $\{x_{n}\}$ converges to a point $x_{\infty} \in \overline{B_{1}}$.  Using the above rescaling we get a sequence $\tilde{u}_{n}$ solving the equation
\begin{equation}
F(D^{2}\tilde{u}_{n}, x_{n} + d_{x_n}N_{u_{n}}^{-1} z ) =0,
\end{equation}
on $B_{N_{n}}(0)$.  Furthermore, for every fixed $k$, the functions $u_{n}, n\geq k$ are defined on $B_{N_{k}}(0)$ and uniformly bounded in $C^{3}(B_{N_{k}}(0))$.  Since $F$ is smooth in both slots, and uniformly elliptic, it follows from the Schauder theory that $\{\tilde{u}_{n}\}_{n\geq k}$ are uniformly bounded in $C^{3,\alpha}(B_{N_{k}/2}(0))$.  Since $N_{n} \rightarrow \infty$ by assumption, we may find a function $\tilde{u} :\mathbb{R}^{n}\rightarrow \mathbb{R}$ and a diagonal subsequence (not relabelled) so that $\{u_{n}\}$ converges uniformly to $\tilde{u}$ in $C^{3,\alpha/2}$ on compact subsets of $\mathbb{R}^{n}$.  In particular, we have $|D^{3}\tilde{u}|(0)=1$.   Furthermore, $\tilde{u}$ satisfies
\[
 F(D^{2}\tilde{u}, x_{\infty}) =0.
\]
Applying the Shauder theory again, we conclude that $\tilde{u} \in C^{\infty}(\mathbb{R}^{n})$.  But, since $F( \cdot, x_{\infty})$ is of twisted type, Lemma~\ref{lem: Lois} implies that $\tilde{u}$ is a quadratic polynomial, which contradicts $|D^{3}\tilde{u}(0)| =1$.

We conclude that $u$ must have $N_{u} <C$, for some $C$ depending on $F$, and $|D^{2}u|_{L^{\infty}(B_{2})}$.  The theorem follows immediately.
\end{proof}

{\bf Acknowledgements}  I would like to thank G. Sz\'ekelyhidi, V.Tosatti and C. Mooney for several helpful conversations.  I am grateful to D.H. Phong and S.-T. Yau for their encouragement and support.

\end{document}